\documentclass{article}

\usepackage[english]{babel}
\usepackage{xcolor}
\usepackage{authblk}

\usepackage[letterpaper,top=2cm,bottom=2cm,left=3cm,right=3cm,marginparwidth=1.75cm]{geometry}

\usepackage[utf8]{inputenc}
\usepackage{mathtools}
\usepackage{amsmath}
\usepackage{amsthm}
\usepackage{algorithm}
\usepackage{algorithmic}
\usepackage{amssymb}
\usepackage{graphicx}
\usepackage[colorlinks=true, allcolors=blue]{hyperref}
\usepackage{tikz}

\usepackage{caption}
\usepackage{subcaption}

\def\calA{{\mathcal A}}  
\def\calD{{\mathcal D}}  
\def\calG{{\mathcal G}}  
  
\def\calM{{\mathcal M}}  
\def\calP{{\mathcal P}}  
  
\def\calV{{\mathcal V}}

\newcommand{\CE}{{\mathsf{CE}}}

\newcommand{\R}{\mathbb{R}}
\newcommand{\N}{\mathbb{N}}
\newcommand{\sfE}{\mathsf{E}}
\newcommand{\sfG}{\mathsf{G}}
\newcommand{\sfL}{\mathsf{L}}
\newcommand{\sfV}{\mathsf{V}}

\DeclarePairedDelimiter{\skp}{\langle}{\rangle}

\usepackage{bm}
\usepackage{hyperref}
\usepackage{enumerate}
\usepackage{enumitem}
\usepackage{mathabx}
\usepackage{fixme}
\usepackage{mathtools}

\newtheorem{theorem}{Theorem}[section]

\newtheorem{proposition}[theorem]{Proposition}

\newtheorem{assumption}[theorem]{Assumption}
\newtheorem{definition}[theorem]{Definition}

\newtheorem{remark}[theorem]{Remark}

\title{Dynamic Optimal Transport with optimal star shaped graphs}
\author[1]{Marcello Carioni}
\author[2]{Juliane Krautz}
\author[2,3]{Jan-F. Pietschmann}
\affil[1]{{
\small
Department of Applied Mathematics, University of Twente, P.O. Box 217, 7500 AE Enschede,
The Netherlands. Email: m.c.carioni@utwente.nl
}}
\affil[2]{{
\small
Universit\"{a}t Augsburg, Institut f\"ur Mathematik, Universit\"{a}tsstra\ss e 12a, 86159 Augsburg, Germany. Emails: \{juliane.krautz, jan-f.pietschmann\}@uni-a.de
}}
\affil[3]{{
\small
Centre for Advanced Analytics and Predictive Sciences (CAAPS), University of Augsburg,
Universit\"{a}tsstr. 12a, 86159 Augsburg, Germany. 
}}
\date{}
\begin{document}
\maketitle
\begin{abstract}
We study an optimal transport problem in a compact convex set $\Omega\subset\R^d$ where bulk transport is coupled to dynamic optimal transport on a metric graph $ \sfG  = (\sfV,\sfE)$ which is embedded in $\Omega$. We prove existence of solutions for fixed graphs. Next, we consider varying graphs, yet only for the case of star-shaped ones. Here, the action functional is augmented by an additional penalty that prevents the edges of the graph to overlap. This allows to preserve the graph topology and thus to rely on standard techniques in Calculus of Variations in order to show existence of minimizers.

\end{abstract}

\section{Introduction}

Metric graphs play an important role in the modeling of real-world phenomena such as gas or road networks, \cite{Fazeny2024,Coclite2005}. In particular, the theory of optimal transport has proven to be a useful tool to analyze transportation on such graphs. In recent years, several related formulations have been studied, see e.g. \cite{Erbar2022}, where the authors provide a complete analysis of optimal transport on metric graphs with Kirchhoff-Neumann conditions at the vertices. Furthermore, 
\cite{BurgerHumpertPietschmann2023} studies a transport metric that allows for mass storage on the vertices. These studies are strongly motivated by and related to a transport distance involving bulk and surface transport that was introduced in \cite{monsaingeon_new_2021}.
Related gradient flows with respect to these distances have also been studied, again in \cite{Erbar2022} and, for the case with vertex dynamics, in \cite{heinze_gradient_2024}. Independently, dynamic optimal transportation with non-linear mobilities was studied in \cite{dolbeault_new_2009}, see also \cite{Lisini2010} for the case of a volume filling mobility.

In the present article, we combine these approaches to a model which consists of transport in a compact and convex domain $\Omega\subset\R^d$, coupled to optimal transport on a metric graph. The introduction of mobilities allows to impose lower and upper bounds on the mass densities, leading to a more versatile model.
The coupled problem is motivated by planning of traffic routes. A particular focus lies on networks where, due to pre-existing infrastructure or connections to other transportation networks, one city is highlighted as a central point. We assume that traveling along the graph is cheaper than traveling in the bulk domain. However, there is an additional cost to entering the graph, that can be thought of as the cost of train tickets or waiting times between different connections.
An example of such a system can be found in French railroad planning, where Paris is associated to the accentuated center of the graph. 
This article extends the results of \cite{carioni2025dynamicoptimaltransportoptimal}, where a coupled dynamic optimal transport problem between bulk and a second domain, which is the graph of a function (a single road) is analyzed. The resulting problem is understood as a dynamic optimal transportation problem together with an additional penalty functional allowing for an optimization over the metric graph as well. 

To introduce our model, we first consider an arbitrary metric graph $\sfG=(\sfV,\sfE,l)$ that is embedded in some convex and compact domain $\Omega\subset\R^d$ with initial and final data given as non-negative Radon measures $\mu_0$ and $\mu_1$ on $\Omega$ and $\rho_0$ and $\rho_1$ on $\sfG$ such that $\mu_0(\Omega) + \rho_0(\sfG) = 1 = \mu_1(\Omega) + \rho_1(\sfG)$. This means that each vertex $v\in\sfV$ can be identified with a point $x_v\in\Omega$ and we think of edges as straight lines connecting these points. For a rigorous definition see Section \ref{notation}. Any measure $\rho_t$ on $\sfG$ can be determined by its values on the edges, thus we identify $\rho_t = (\rho_t^e)_{e\in\sfE}$.
Now, the evolution of mass on the coupled system can be expressed, formally, by solutions of the following system of continuity equations:
\begin{equation}\label{eq:FormalSystem}
    \left\{
    \begin{aligned}
        \partial_t \mu_t + \nabla\cdot J_t &= 0 &&\text{in } \Omega\\
        J_t\cdot n^e_\Omega &= f_t^e &&\text{for } e\in\sfE\\
        \partial_t \rho^e_t + \nabla\cdot V^e_t &= f_t^e &&\text{for } e\in\sfE\\
        V^e_t  n^v_e &= G(v) = {J_t}_{|\{x_v\}}\cdot \tau_e &&\text{for } v \in \sfV_{\rm out} \\
        \sum_{e\in \sfE(v)} V^e_t  n^v_e &= 0 &&\text{for } v\in\sfV\setminus\sfV_{\rm out} 
    \end{aligned}\right.
\end{equation}
where $n^e_\Omega\in\mathbb{S}^{n-1}$ denotes the normal to the edges $e\in\sfE$, $\tau_e\in\mathbb{S}^{n-1}$ denotes the tangent to the embedded edge in $\Omega$ and $n^v_e\in\{0,\pm1\}$ is defined as in \eqref{eq:EdgeNormal}, prescribing an orientation to each edge. For an illustration see \autoref{fig:GraphNormalTang}. In particular, we impose a non-homogeneous continuity equation on each edge, coupling them in the interior vertices $\sfV\setminus\sfV_{\rm out}$ with homogeneous Neumann conditions, allowing for conservation of mass. For the outer vertices $\sfV_{\rm out}\subset \sfV$, we allow an additional exchange of mass with the bulk region $\Omega$. We denote by $\CE(\sfG)$ the set of weak solutions to \eqref{eq:FormalSystem}.
Now, the evolution of this system is governed by quadratic costs and non-linear mobilities imposing bounds on the mass measures $\mu_t$ and $\rho_t$. The resulting (formal) action functional reads as
\begin{align*}
    \calA_\sfG(\mu_t, J_t, \rho_t, V_t, f_t, G_t) =& \int_\Omega \frac{|J_t|^2}{m_\Omega(\mu_t)}  \, d\lambda_\Omega + \sum\limits_{e\in \sfE} \left(\int_e  \alpha_1 \frac{|V^e_t|^2}{m_\sfG(\rho^e_t)}  + \alpha_2 \frac{|f^e_t|^2}{m_\sfG(\rho^e_t)}  \, d\lambda_e \right)\\
    &+ \alpha_3 \sum\limits_{v\in\sfV_{\rm out}} \sum\limits_{e\in\sfE(v)} \int_{\{x_{v}\}} \frac{|G^e_t(x_v)|^2}{m_\sfG({\rho^e_t}_{|x_v})} \, d\lambda_{v} ,
\end{align*}
where $m_\Omega$ and $m_\sfG$ are mobility functions and $\lambda_\Omega$, $\lambda_e$ and $\lambda_v$ are fixed reference measures, necessary as the mobilities lead to a loss of $1$-homogenity of the functional. We then consider the dynamic problem 
\begin{align*}
    \inf_{(\mu_t, J_t, \rho_t, V_t, f_t, G_t)\in\CE(\sfG)} \calA_\sfG(\mu_t, J_t, \rho_t, V_t, f_t, G_t).
\end{align*}
We refer to Section \ref{fixedgraph} for a rigorous definition of weak solutions to the coupled system \eqref{eq:FormalSystem} and for the resulting dynamic problem.

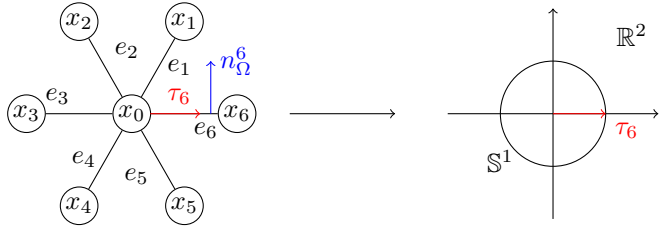
\begin{figure}
    \centering
    \begin{tikzpicture}[scale=.7]
    \node[circle, draw=black, minimum size=5mm,inner sep=0pt] (C) at (0,0) {$x_0$};

    \node[circle, draw=black, minimum size=5mm,inner sep=0pt] (A) at (1,1.75) {$x_1$};
    \node[circle, draw=black, minimum size=5mm,inner sep=0pt] (B) at (2,0) {$x_6$};
    \node[circle, draw=black, minimum size=5mm,inner sep=0pt] (D) at (1,-1.75) {$x_5$};
    \node[circle, draw=black, minimum size=5mm,inner sep=0pt] (E) at (-1,-1.75) {$x_4$};
    \node[circle, draw=black, minimum size=5mm,inner sep=0pt] (F) at (-2,0) {$x_3$};
    \node[circle, draw=black, minimum size=5mm,inner sep=0pt] (G) at (-1,1.75) {$x_2$};

    \draw (C) -- (A) node[midway, right] {$e_1$};
    \draw (C) -- (B) node[midway, below right] {$e_6$};
    \draw (C) -- (D) node[midway, below left] {$e_5$};
    \draw (C) -- (E) node[midway, left] {$e_4$};
    \draw (C) -- (F) node[midway, above left] {$e_3$};
    \draw (C) -- (G) node[midway, above right] {$e_2$};

    \draw[->, blue] (1.5,0) -- (1.5,1) node[right] {$n_\Omega^6$};
    \draw[->, red] (0.35,0) -- (1.3,0) node[above left] {$\tau_6$};

    \draw[->] (3,0) -- (5,0);

    \draw[->] (8,-2) -- (8,2);
    \draw[->] (6,0) -- (10,0);
    \node (R) at (9.5,1.5) {$\R^2$};
    \draw (8,0) circle (1);
    \node (S) at (7,-.9) {$\mathbb{S}^1$};
    \draw[->, red] (8,0) -- (9,0) node[below right] {$\tau_6$};
     
    \end{tikzpicture}
    \caption{A star-shaped metric graph embedded into $\R^2$ with $\sfV_{\rm out} = \{v_1,\ldots,v_6\} = \sfV\setminus \{v_0\}$, normal $n_\Omega^6 = (0,1)^T$ and tangent $\tau_6 = (1,0)^T$.}
    \label{fig:GraphNormalTang}
\end{figure}

In order to account for varying graphs in the second part of this work, we restrict our analysis to the case of star-shaped metric graphs. For such, there exists a unique central vertex $v_0\in\sfV$ which is contained in all edges and we define $\sfV_{\rm out} = \sfV\setminus\{v_0\}$. Thus, all edges can be written as $e=(v_0,v)$ for some $v\in \sfV_{\rm out}$. 
Throughout minimization, we fix the connectivity information and only optimize over the placement of the vertices. For this to result in a well-defined metric graph, we need to prevent vertices from colliding and edges from crossing. To this end we introduce the additional penalty functional 
\begin{align*}
    R(x_{v_0},\ldots,x_{N_\sfV-1}) := \sum\limits_{i=1}^{N_\sfV-1} \left( \frac{1}{| x_{v_i} - x_{v_0} |} + \frac{1}{2}\sum\limits_{\genfrac{}{}{0pt}{}{j=1}{j\neq i}}^{N_\sfV-1} \frac{1}{\alpha_{x_{v_i},x_{v_j}}}\right)
\end{align*}
where $N_V = | \sfV|$ and $\alpha_{x_{v_i},x_{v_j}}$ is the unsigned angle between the edges $(v_0, v_i)$ and $(v_0,v_j)$ at the central vertex $v_0$. The augmented problem now reads 
\begin{align*}
    \inf\limits_{\sfG} \inf\limits_{\CE(\sfG)} \int_0^1 \calA_\sfG(\mu_t, J_t, \rho_t, V_t, f_t, G_t) \, dt + c R(x_{v_0}, x_{v_1}, \ldots, x_{v_{N_\sfV-1}})
\end{align*}
and its rigorous definition is given in Section~\ref{varyinggraphs}. 

The article is structured as follows. In Section \ref{notation} we introduce some notation as well as measures on metric graphs. In Section \ref{fixedgraph} we define the notion of weak solutions to the coupled system as well as the dynamic formulation. Additionally, we prove existence of minimizers. 
In the last part, in Section \ref{varyinggraphs}, we analyze the dynamic problem for varying graphs where we restrict our considerations to star-shaped metric graphs and optimize over the placement of vertices. Again, we prove existence of minimizers.

\section{Notation}\label{notation}

In order to rigorously define the coupled dynamic system, we start by introducing some notation. 

Let $\Omega\subset\R^d$ be a compact and convex subset of $\R^d$ and $\sfG  = (\sfV,\sfE)$ be a combinatorial graph, where $\sfV$ denotes the \emph{set of nodes} and $\sfE$ is the \emph{set of edges} $e = (v,w)$ for $v,w\in \sfV $. The number of nodes is denoted by $N_\sfV := |\sfV|$ and the number of edges by $N_\sfE := | \sfE|$. We define $\sfE(v) := \{e\in\sfE\, | \, v\in e\}$ as the set of edges containing $v\in\sfV$. Moreover, we introduce the set of \emph{outer vertices} $\sfV_{\rm out} := \{v\in\sfV \, | \, |\sfE(v)| = 1\}$ and the set of \emph{outer edges} $\sfE_{\rm out} := \{e\in\sfE \, | \, \exists~ v\in \sfV_{\rm out} \text{ s.t. } v\in e \}$.
With an additional map $l:\sfE\to(0,+\infty)$, associating a positive length $l_e$ to each edge $e$ of the graph, the combinatorial graph turns into a \emph{metric graph} $\sfG=(\sfV,\sfE,l)$. By defining the outer normal 
\begin{align}\label{eq:EdgeNormal}
    n_e^v := \begin{cases}
        -1 &:  e=(v,w) \\ 0 & :   v\notin e\\ +1 &:  e=(w,v)
    \end{cases}
\end{align}
and therefore fixing an orientation, we can identify edges with closed intervals $[0,l_e]$. In order to rigorously define the coupling between a domain and a metric graph, we need to embed the graph. 
\begin{definition}\label{def:Embedding}
    Let $\sfG=(\sfV, \sfE)$ be a combinatorial graph and $\Omega\subset\R^d$ be a compact and convex set. A collection of distinct points $x_v\in\Omega$ for $v\in\sfV$ is called an \emph{embedding} of the graph, defining edges as straight lines, and we say that the graph $\sfG$ is \emph{embedded into $\Omega$} if the embedding has no intersecting edges. 
    We say that a metric graph $\sfG = (\sfV, \sfE, l)$ is \emph{embedded into $\Omega$} if it is embedded as a combinatorial graph and if $|x_w - x_v| = l_e$ for all $e\in \sfE $.
\end{definition}

 For an embedded metric graph we denote the tangent vector of an edge $e=(v,w)\in\sfE$ by
\begin{align}
    \tau_e := \frac{x_w-x_v}{l_e} = \frac{x_w-x_v}{|x_w-x_v|} 
\end{align}
and we write $[x_v,x_w] = \{x_v + t \tau_e \, | \, t\in[0,l_e]\}\subset\Omega$ for the interval defined by an edge $e=(v,w)$ of an embedded metric graph. Abusing notation, we identify $e=[x_v,x_w]$ throughout the article.

Functions on metric graphs are defined by their restriction to each edge. We introduce the sets 
\begin{align}
\Omega_\sfE := \bigsqcup_{e=(v,w)\in\sfE} [x_v, x_w] ~
\text{,} \quad \sfL := \bigsqcup_{e=(v,w)\in\sfE} [0,1] \quad \text{and} \quad \Omega_\sfV := \bigsqcup\limits_{v\in\sfV_{\rm out}} \{x_v\}
\end{align}
given as disjoint unions.
In order to guarantee a metric structure on the metric graph and therefore identifying it with a metric space, we define the new sets
\begin{align*}
    \Omega_\sfG := \Omega_\sfE\big/\sim_\sfE \quad \text{and} \quad \sfL_\sfG := \sfL\big/\sim_\sfL,
\end{align*}
where $\sim_\sfE$ denotes the equivalence relation that identifies the same vertex on different edges and $\sim_\sfL$ is the corresponding identification in $\sfL$. Note that in $\Omega_\sfG$ each vertex $v\in\sfV$ can be identified uniquely, whereas there are multiple copies in $\Omega_\sfE$, one for each edge $e\in\sfE(v)$.
Any map on a metric graph is given as $\varphi=(\varphi^e)_{e\in \sfE}:\Omega_\sfE\to\R$ or $\Tilde{\varphi}=(\Tilde{\varphi}^e)_{e\in \sfE}:\sfL\to\R$. Both can be related using the transformation maps $\gamma_e(s) := x_v + s l_e\tau_e$ and $\Tilde{\gamma}_e(x) = \frac{|x-x_v|}{l_e}$ for $e=(v,w)$ as 
\begin{align}
    \Tilde{\varphi}^e(s) = \varphi^e(\gamma_e(s)) \quad \text{and} \quad \varphi^e(x) = \Tilde{\varphi}^e(\Tilde{\gamma}_e(x)).
\end{align}
Similar definitions can be given for $\Omega_\sfG$ and $\sfL_\sfG$.
Finally, we define $C^1([0,1]\times L) := C^1([0,1];C(\sfL_\sfG))\cap C([0,1]; C^1(\sfL))$, i.e. functions that are continuous in each vertex having derivatives on each edge.

We call a (metric or combinatorial) graph \emph{star-shaped} if there exists an accentuated node $v_0\in \sfV $ such that $ E  = \{(v_0,v)\,|\, v\in  \sfV \setminus\{v_0\}\}$. The node $v_0$ is called \emph{center of the graph}. In particular it holds that $\sfE = \sfE(v_0) = \sfE_{\rm out}$ and $\sfV_{\rm out} = \sfV\setminus\{v_0\}$. 

In order to rigorously formulate the coupling between different domains, we need to extend measures defined on the graph to measures on $\Omega$. 
Let $\calM(U,\mathbb{R}^N)$ be the set of Radon measures on $U \subset \R^d$ with values in $\mathbb{R}^N$ and $\calM_+(U, \mathbb{R}^N)$ the set of measures in $\calM(U,\R^N)$ that are non-negative. Moreover, we denote by  $\calM_{[0,1]}(U, \R^N)$ a Borel measurable family of measures in $U$ indexed by $t\in[0,1]$ with values in $\R^N$. 
We also consider measures with values in the tangent space of some edge $e \in \sfE$. In particular, given $x \in e$ the measures $V^e\in \calM(e, T_xe)$ and $G^e(x_v)\in\calM(\{x_v\}, T_{x_v}e)$ will occur in our formulation, where $T_xe$ is the tangent space to $e$ in $x$. Note that we can represent such measures as
\begin{align}
    V^e = \tau_e \calV^e \quad \text{or} \quad G^e(x_v) = \tau_e \calG^e(x_v) 
\end{align}
with $\calV^e\in \calM(e,\mathbb{R})$,  $\calG^e(x_v) \in \calM(\{x_v\},\mathbb{R})$. For any measure $\rho^e\in\calM(e,\R)$ we define its zero extension into $\Omega$ by duality as the measure $\Bar\rho^e \in \calM(\Omega, \R)$ such that
\begin{align}
    \int_\Omega \phi \, d\Bar\rho^e = \int_e \phi \, d\rho^e
\end{align}
for all $\phi\in C(\Omega)$. For measures $\rho\in\calM(\Omega_\sfE,\R)$ we define their extension as $\Bar\rho = (\Bar\rho^e)_{e\in \sfE}$ and for $\Tilde{\rho}\in \calM(\sfL,\R)$ as $\overline{\Tilde{\rho}} = (\overline{\Tilde{\rho}^e\circ\Tilde{\gamma}_e})_{e\in \sfE}$.
Measures $\rho\in\calM(\Omega_\sfE, \R)$ can be identified with measures $\Tilde{\rho}\in \calM(\sfL,\R)$ through the push-forward by the maps $\gamma_e$ for $e \in E$. In particular, when it will be necessary to underline the dependence on $\gamma_e$ we will write $\tilde \rho = \rho \circ \gamma_e$.
Similar notions of extension and push-forward  operations can be defined for measures defined on $\Omega_\sfG$ and $\sfL_\sfG$ with values in $\R^N$ for $N \in \mathbb{N}$.

With these notions, we are able to introduce the space of admissible tuples $(\mu_t, J_t, \rho_t, V_t, f_t, G_t)$ as the space $\calD_{\rm adm}(\sfG)$ defined by 
\begin{align*}
    \calD_{\rm adm}(\sfG) := &\calM_{[0,1]}(\Omega, \R)\times\calM_{[0,1]}(\Omega, \R^d)\times  \calM_{[0,1]}(\Omega_\sfG, \R)\times  \calM_{[0,1]}(\Omega_\sfE, T_xe)\\
    &\times \calM_{[0,1]}( \Omega_\sfE, \R)\times  \calM_{[0,1]}(\Omega_\sfV, T_{x_v}e).
\end{align*}
Admissible initial or final measures are given as elements of the set 
\begin{align}
    \calP_{\rm adm}(\sfG) := \left\{(\mu, \rho) \in \calM_+(\Omega)\times\calM_+(\Omega_\sfG) \, | \, \mu(\Omega) + \rho(\Omega_\sfG) = 1\right\}.
\end{align}

\section{Fixed Graph}\label{fixedgraph}

In this section we consider a fixed metric graph $\sfG=(\sfV,\sfE,l)$, embedded in $\Omega\subset\R^d$ compact and convex. We rigorously define the dynamic problem and show existence of solutions to the formal transport problem \eqref{eq:FormalSystem} with minimal dynamic costs. 

\subsection{Continuity equation}
We rigorously define the coupled continuity equations and show mass conservation of the entire system. 

\begin{definition}\label{WeakFormCE}
    We say that a tuple $(\mu_t, J_t, \rho_t, V_t, f_t, G_t)\in\calD_{\rm adm}(\sfG)$ satisfies the coupled continuity equations for admissible initial and final data $(\mu_0, \rho_0), (\mu_1, \rho_1) \in \calP_{\rm adm}(\sfG)$ if 
    \begin{align*}
        &\int_0^1 \int_\Omega \partial_t \phi_t(x) \, d\mu_t dt + \int_0^1 \int_\Omega \nabla \phi_t(x) \, d J_t dt 
        - \sum\limits_{e\in\sfE}\int_0^1 \int_0^1 \phi_t(\gamma_e(s)) \, d \Tilde{f}^e_t dt 
       \\
       & \qquad \qquad + \sum\limits_{v\in\sfV_{\rm out}}\sum\limits_{e\in \sfE(v)}\int_0^1 \int_{\{x_v\}} \phi_t(x_v)n^v_e\, d\Tilde{\calG}^e_t(x_v) dt 
        = \int_\Omega \phi_1(x)\, d\mu_0 - \int_\Omega \phi_0(x)\, d\mu_1
    \end{align*}
    for all $\phi_t\in C^1([0,1]\times\Omega)$
    and 
    \begin{align*}
        &\sum\limits_{e\in \sfE}\left(\int_0^1 \int_0^1 \partial_t \psi_t^e(s) \, d\Tilde{\rho}^e_t dt + \int_0^1\int_0^1 \nabla\psi_t^e(s) l_e \, d\Tilde{\calV}^e_t dt + \int_0^1\int_0^1 \psi_t^e(s) \, d\Tilde{f}^e_t dt\right) 
        \\
        & \qquad \qquad  -\sum\limits_{v\in\sfV_{\rm out}}\sum\limits_{e\in \sfE(v)}\int_0^1 \int_{\{x_v\}} \psi^e_t(x_v)n^v_e\, d\Tilde{\calG}^e_t(x_v) dt
        = \sum\limits_{e\in \sfE} \left(\int_0^1 \psi_1^e(s)\, d\Tilde{\rho}^e_0 - \int_0^1 \psi_0^e(s)\, d\Tilde{\rho}^e_1 \right)
    \end{align*}
    for all $\psi_t=(\psi_t^e)_{e\in \sfE}\in C^1([0,1]\times L)$ such that $\psi^e_t(\Tilde{\gamma}_e(x_v)) =: \psi_t(v)$ for all $e\in\sfE(v)$, meaning that the map is continuous over vertices. We denote by $\CE(\sfG)$ the set of such solutions on the graph $\sfG$.
\end{definition}

We can show that solutions to this system satisfy a global continuity equation and therefore are mass preserving.

\begin{proposition}
    Suppose that $(\mu_t, J_t, \rho_t, V_t, f_t, G_t)\in\CE(\sfG)$. Given $\eta_t := \mu_t + \sum_{e\in \sfE}\Bar \rho^e_t$ and $W_t := J_t + \sum_{e\in \sfE}\Bar V^e_t$ it holds that 
    \begin{align}
        \partial_t \eta_t + \nabla\cdot W_t = 0
    \end{align}
    weakly with initial data $\mu_0 + \sum_{e\in \sfE}\Bar \rho^e_0$ and final data $\mu_1 + \sum_{e\in \sfE}\Bar \rho^e_1$.
\end{proposition}

\begin{proof}
    Let $\phi\in C^1([0,1]\times\Omega)$ be an arbitrary test function. For each edge $e \in \sfE$ we define $\psi^e_t = \phi_t\circ \gamma_e\in C^1([0,1]\times[0,1])$ which, by construction, induces the admissible test function $\psi_t = (\psi_t^e)_{e\in \sfE} \in C^1([0,1]\times \sfL_\sfG)$. Moreover, $\frac{d}{ds}\psi_t^e(s) = \nabla\phi_t(\gamma_e(s))\cdot l_e\tau_e$. We obtain 
    \begin{align*}
        &\int_0^1 \int_\Omega \partial_t \phi_t(x) \, d\eta_t dt + \int_0^1 \int_\Omega \nabla \phi_t(x) \, d W_t dt \\
        =& \int_0^1 \int_\Omega \partial_t \phi_t(x) \, d\mu_t dt + \int_0^1 \int_\Omega \nabla \phi_t(x) \, d J_t dt + \sum\limits_{e\in \sfE} \left[ \int_0^1 \int_0^1 \partial_t \psi_t(s) \, d\Tilde{\rho}^e_t dt + \int_0^1 \int_0^1 \nabla \psi_t(s) \cdot l_e\tau_e\, d  \Tilde{V}^e_t dt \right]\\
        =& \int_\Omega \phi_1(x) \, d\mu_1 - \int_\Omega \phi_0(x) \, d\mu_0 + \sum\limits_{e\in \sfE} \left( \int_0^1 \psi^e_1(s) \, d\tilde{\rho}^e_1 - \int_0^1 \psi^e_0(s) \, d\Tilde{\rho}^e_0\right)
    \end{align*}
    where the last equality follows by substituting the weak formulation from of Definition \ref{WeakFormCE}, thus proving the statement.
\end{proof}

\subsection{Dynamic formulation}

We will define the variational formulation that governs the evolution of the measures $(\mu_t, J_t, \rho_t, V_t, f_t, G_t)\in\CE(\sfG)$ over time. It is given as a generalized kinetic energy functional with an additional mobility function defining upper and lower bounds on the mass densities.

\begin{definition}[Admissible mobilities]
    We call a function $m:[0,+\infty)\to[0,+\infty)\cup\{-\infty\}$ \emph{admissible mobility} if it is an upper semi-continuous and concave function with ${\rm{int}}({\rm{dom}}(m)) = (a,b)$ for $0\leq a<b$ and $m(z)>0$ for all $z\in (a,b)$.
\end{definition}

We are now able to rigorously introduce the dynamic problem. 

\begin{definition}
    We define the following variational problem
    \begin{align}\label{VarFormFixedG}\tag{BB}
        \inf\limits_{\CE(\sfG)} \int_0^1 \calA_\sfG(\mu_t, J_t, \rho_t, V_t, f_t, G_t) \, dt,
    \end{align}
    where
    \begin{align*}
        &\calA_\sfG(\mu_t, J_t, \rho_t, V_t, f_t, G_t) \\
        =& \int_\Omega \Psi_\Omega\left(\frac{d \mu_t}{d \lambda_\Omega}, \frac{d J_t}{d \lambda_\Omega} \right) \, d\lambda_\Omega 
        + \sum\limits_{e\in \sfE} \left(\int_e  \alpha_1  \Psi_\sfG\left(\frac{d \rho^e_t}{d \lambda_e}, \frac{d \calV^e_t}{d \lambda_e} \right) + \alpha_2 \Psi_\sfG\left(\frac{d \rho^e_t}{d \lambda_e}, \frac{d f^e_t}{d \lambda_e} \right) \, d\lambda_e \right)\\
        +& \alpha_3 \sum\limits_{v\in\sfV_{\rm out}} \sum\limits_{e\in\sfE(v)} \int_{\{x_{v}\}} \Psi_\sfG\left(\frac{d \rho^e_t |_{x_v}}{d \lambda_{v}}, \frac{d \calG^e_t(x_v)}{d \lambda_{v}} \right)\, d\lambda_{v} 
    \end{align*}
    for $\alpha_1, \alpha_2, \alpha_3>0$ describing the cost of transport along the graph ($\alpha_1$) and entering or leaving the graph ($\alpha_2, \alpha_3$) respectively. 
    Here, $\lambda_\Omega\in\calM_+(\Omega)$, $\lambda_e\in\calM_+(e)$ and $\lambda_v\in\calM_+(\{x_v\})$ are non-negative Radon-measures such that $\mu_t, J_t \ll\lambda_\Omega$,  $\rho_t^e, \calV_t^e, f_t^e \ll\lambda_e$ and $\rho_t^e|_{x_v}, \calG_t^e(x_v) \ll \lambda_v$ for almost all $t\in[0,1]$. The action functionals are defined as 
    \begin{align}
        \Psi_\Omega(u,v) := \begin{cases}
            \frac{|v|^2}{m_\Omega(u)} &: m_\Omega(u)\neq 0\\
            0 &: m_\Omega(u) = 0 = v\\
            +\infty &:\text{else}
        \end{cases}
        \quad \text{and} \quad
        \Psi_\sfG(u,v) := \begin{cases}
            \frac{|v|^2}{m_\sfG(u)} &: m_\sfG(u)\neq 0\\
            0 &: m_\sfG(u) = 0 = v\\
            +\infty &:\text{else}
        \end{cases}
    \end{align}
    for admissible mobility functions $m_\Omega$ and $m_\sfG$.
\end{definition}

\begin{remark}\label{lscFixedG}
    As shown in \cite[Theorem 2.1]{dolbeault_new_2009}, functionals of the form \ref{VarFormFixedG} are convex and lower semi-continuous with respect to the weak convergence of measures.
\end{remark}

Boundedness of the action $\calA_\sfG$ allows proving further regularity results for the mass densities $\mu_t$ and $\rho_t^e$, $e\in \sfE$ following the same arguments as in \cite[Theorem 3.5]{monsaingeon_new_2021}.

\begin{proposition}\label{densityBnd}
    For $(\mu_t, J_t, \rho_t, V_t, f_t, G_t)\in\CE(\sfG)$ with $\int_0^1 \calA_\sfG(\mu_t, J_t, \rho_t, V_t, f_t, G_t) \, dt< + \infty$ the measures $\mu_t$ and $\rho_t$ admit a weakly continuous representative. Suppose further that ${\rm{int}}({\rm{dom}}(m_\Omega)) = (a_\Omega, b_\Omega)$ and ${\rm{int}}({\rm{dom}}(m_\sfG)) = (a_\sfG, b_\sfG)$ and denote by $\mu_t$, $\rho^e_t$ the densities of the measures with respect to $\lambda_\Omega$ and $\lambda_\sfG$. It holds that
    \begin{align}
        a_\Omega \leq \mu_t \leq b_\Omega, \, \lambda_\Omega-{\rm{a.e.}} \quad \text{and} \quad a_\sfG \leq \rho^e_t \leq b_\sfG, \, \lambda_\sfG-\rm{a.e.}
    \end{align}
    for every $t\in[0,1]$ and every $e\in \sfE$. 
\end{proposition}

\subsection{Existence of minimizers for a fixed graph}

In this section we will show well-posedness of \ref{VarFormFixedG} using the direct method of calculus of variations. From Remark \ref{lscFixedG} we infer lower semicontinuity. Compactness will follow from a Hölder estimate. 

\begin{proposition}[Hölder estimate]\label{HoelderFixedG}
    Let us suppose that $\int_0^1 \calA_\sfG(\mu_t, J_t, \rho_t, V_t, f_t, G_t) \, dt \leq M$ for a given $(\mu_t, J_t, \rho_t, V_t, f_t, G_t)\in \CE(\sfG)$ and a constant $M\geq 0$. Then, it holds that 
    \begin{align}
        \| \mu_t - \mu_s\|_{C^1(\Omega)^*} &\leq C(b_\Omega, b_\sfG, M) |t-s|^\frac{1}{2}\\
        \sum\limits_{e\in \sfE}\| \rho^e_t - \rho^e_s\|_{C^1(e)^*} &\leq C(b_\Omega, b_\sfG, M) |t-s|^\frac{1}{2}
    \end{align}
    for a constant $C(b_\Omega, b_\sfG, M)>0$ and for all $t,s\in[0,1]$.
\end{proposition}

\begin{proof}
    The statement can be proven analogously to \cite[Proposition 2.10]{carioni2025dynamicoptimaltransportoptimal} following arguments from \cite[Proposition 3.5]{monsaingeon_new_2021}. 
\end{proof}
From the previously stated Hölder estimates we can infer compactness. 

\begin{proposition}[Compactness]\label{CompFixedG}
    Given $(\mu^n_t, J^n_t, \rho^n_t, V^n_t, f^n_t, G^n_t)\in \CE(\sfG)$, $n\in \N$ and $M\geq 0$ such that 
    \begin{align}
        \sup\limits_{n\in\N} \int_0^1 \calA_\sfG(\mu^n_t, J^n_t, \rho^n_t, V^n_t, f^n_t, G^n_t) \, dt \leq M,
    \end{align}
    there exists $(\mu_t, J_t, \rho_t, V_t, f_t, G_t)\in \CE(\sfG)$ such that, up to a subsequence, the following convergences hold:
    \begin{enumerate}[label=\roman*)]
        \item $\mu_t^n \rightharpoonup \mu_t$ weakly in $\calM(\Omega)$ for all $t\in[0,1]$ \label{mu}
        \item ${\rho}_t^n \rightharpoonup \rho_t$ weakly in $\calM(\Omega_\sfG)$ for all $t\in[0,1]$ \label{rho} 
        \item $J_t^n \rightharpoonup J_t$ weakly in $\calM([0,1]\times\Omega, \R^d)$ \label{J}
        \item $\calV_t^n \rightharpoonup\calV_t$ in $\calM([0,1]\times \Omega_\sfE)$\label{V}  
        \item $f^n_t \rightharpoonup f_t$ in $\calM([0,1]\times \Omega_\sfE)$ \label{f}
        \item $\calG^n_t \rightharpoonup \calG_t$ in $\calM\left([0,1]\times\Omega_{\rm V}\right)$.\label{G}
    \end{enumerate}
\end{proposition}

\begin{proof} 
    The proof of \ref{mu} and \ref{rho} follows from Proposition \ref{HoelderFixedG}, Proposition \ref{densityBnd} and an application of a generalized Arzela-Ascoli theorem, see \cite[Proposition 3.3.1]{ambrosio_gradient_2008}.
    The remaining convergence results follow by standard estimates on the total variation of the measures and Prokhorovs theorem carried out in \cite[Theorem 2.11]{carioni2025dynamicoptimaltransportoptimal}.
    Admissibility of the limit is a direct consequence of the weak convergence and disintegration properties.
\end{proof}

We are now in a position to show well-posedness of \ref{VarFormFixedG}. 

\begin{theorem}
    Suppose that \ref{VarFormFixedG} is finite. Then, there exists a minimizer $(\mu_t, J_t, \rho_t, V_t, f_t, G_t)\in \CE(\sfG)$.
\end{theorem}

\begin{proof}
    This is a direct consequence of Remark \ref{lscFixedG}, Proposition \ref{CompFixedG} and the direct method of calculus of variations.
\end{proof}

\section{Varying Graph}\label{varyinggraphs}

In this section we analyze the case of varying graphs restricting our considerations to star-shaped metric graphs. From now on, we suppose that $\sfG = (\sfV, \sfE, l)$ is star-shaped with central vertex $v_0\in\sfV$. 
We keep the connectivity of the graph fixed and we optimize the graph by varying the placement of the vertices inside of $\Omega$. For given initial and final data we are therefore looking for an optimal embedding of the combinatorial graph. Note that we allow the edge length to vary as well.
In order for the continuity equation to be well-defined throughout the minimization, we need to prevent overlapping edges. For this reason, we impose an additional constraint on the unsigned angle $\alpha_{x_v,x_w}\in [0,2\pi)$ between any pair of edges $(v_0,v)$ and $(v_0,w)$ at the common central node $v_0$. The angle can be calculated as
\begin{align}
    \alpha_{x_v,x_w} = \left|\arccos \frac{\skp{x_v-x_{v_0},x_w-x_{v_0}}}{|x_v-x_{v_0}||x_w-x_{v_0}|} \right| 
\end{align}
and we approach an overlap of two edges if this angle is small.
However, this does not prevent the nodes from collapsing in the center $x_{v_0}$. Therefore, we add an additional coloumb-type potential in order to obtain lower bounds on the distances $|x_v-x_{v_0}|$ for $v\in \sfV$. 
The resulting penalty functional reads
\begin{align}
    R(x_{v_0}, x_{v_1}, \ldots, x_{v_{N_\sfV-1}}) := \sum\limits_{i=1}^{N_\sfV-1} \left( \frac{1}{| x_{v_i} - x_{v_0} |} + \frac{1}{2}\sum\limits_{\genfrac{}{}{0pt}{}{j=1}{j\neq i}}^{N_\sfV-1} \frac{1}{\alpha_{x_{v_i},x_{v_j}}}\right)
\end{align}
inducing the variational problem
\begin{align}\label{VarFormVarG}\tag{BB$_\sfG$}
    \inf\limits_{\sfG \text{ embedded}} \inf\limits_{\CE(\sfG)} \int_0^1 \calA_\sfG(\mu_t, J_t, \rho_t, V_t, f_t, G_t) \, dt + c R(x_{v_0}, x_{v_1}, \ldots, x_{v_{N_\sfV-1}})
\end{align}
as $c>0$ allows for different scalings.
The following properties are direct consequences of the definition of the penalty functional. 

\begin{proposition}\label{penalty_prop}
    Any representation $x_{v_0}, x_{v_1}, \ldots, x_{v_{N_\sfV-1}}\in \Omega$ of a star-shaped metric graph with $R(x_{v_0}, x_{v_1}, \ldots, x_{v_{N_\sfV-1}}) <  + \infty$ is embedded in the sense of Definition \ref{def:Embedding}. In particular, there is no overlap between different edges and there exists a lower bound $d>0$ such that $|x_{v_i}-x_{v_j}|\geq d$ for all $i,j\in\{0,\ldots,N_\sfV-1\}$, $i\neq j$.
\end{proposition}

\begin{remark}\label{lscVarG}
    If the penalty functional is finite, it is continuous with respect to the strong norm convergence of the vertices. In particular, it is lower semi-continuous. 
\end{remark}

When varying the graph, we need to account for changes in the initial and final data. To do so, let $\eta_0, \eta_1 \in \calP(\Omega)$ be given. 
We define $\rho_k := {\eta_k}_{|\Omega_\sfG}$
and $\mu_k = \eta_k - \Bar \rho_k$ for $k=0,1$. Since
\begin{align}
    \mu_k(\Omega) + \sum\limits_{e\in \sfE} \rho_k^e(e) &= \eta_k(\Omega\setminus\Omega_\sfG) + \eta_k(\Omega_\sfG) = 1
\end{align}
we constructed a pair of admissible initial and final data $(\mu_k, \rho_k)\in \calP_{\rm adm}(\sfG)$, $k=0,1$. In this way, initial and final data can be defined independently of the graph.

For the minimization we need to assume compatibility conditions for the mobilities and reference measures along varying graphs as well. 

\begin{assumption}
We make the following additional assumptions:
\begin{enumerate}
    \item The mobilities $m_\Omega$ and $m_\sfG$ do not depend on the graph.
    \item It holds that $\sup\limits_{\sfG \text{ embedded}} \sum_{e \in \sfE}\| \lambda_e\|_{\calM(e)}<+\infty$ and $\sup\limits_{\sfG \text{ embedded}} \sum_{v \in \sfV}\| \lambda_v\|_{\calM(\{x_v\})}<+\infty$.
    \item For $x_{v_i}^n\to x_{v_i}$ as $n\to\infty$, $i\in\{0,\ldots,N_\sfV-1\}$ it holds that $\lambda_{e^n}\rightharpoonup\lambda_e$ and $\lambda_{v^n}\rightharpoonup\lambda_v$.
\end{enumerate}
\end{assumption}

Under these assumptions, we are able to show a similar compactness result to \autoref{CompFixedG}. 

\begin{proposition}[Compactness for varying graphs]\label{CompVarG}
    Given a family of embedded, star-shaped metric graphs $\sfG^n=(\sfV^n,\sfE^n)$ with nodes $x_{v_i}^n\to x_{v_i}$ as $n\to\infty$, $i\in\{0,\ldots,N_\sfV-1\}$ and measures $(\mu^n_t, J^n_t, \rho^n_t, V^n_t, f^n_t, G^n_t)\in \CE(\sfG^n)$ such that 
    \begin{align}
        \sup\limits_{n\in\N}  \int_0^1 \calA_\sfG(\mu^n_t, J^n_t, \rho^n_t, V^n_t, f^n_t, G^n_t) \, dt + c R(x^n_{v_0}, x^n_{v_1}, \ldots, x^n_{v_{N_\sfV-1}}) < + \infty
    \end{align}
    there exist measures $(\mu_t, J_t, \rho_t, V_t, f_t, G_t)\in \CE(\sfG)$ such that, up to subsequences
    \begin{enumerate}[label=\roman*)]
        \item $\mu_t^n \rightharpoonup \mu_t$ weakly in $\calM(\Omega)$ for all $t\in[0,1]$ \label{vmu}
        \item $\rho_t^n\circ \gamma^n \rightharpoonup \rho_t\circ \gamma$ weakly in $\calM(\sfL_\sfG)$ for all $t\in[0,1]$ \label{vrho} 
        \item $J_t^n \rightharpoonup J_t$ weakly in $\calM([0,1]\times\Omega, \R^d)$ \label{vJ}
        \item $\calV_t^n\circ \gamma^n \rightharpoonup\calV\circ \gamma$ in $\calM([0,1]\times \sfL)$\label{vV}  
        \item $f_t^n\circ \gamma^n \rightharpoonup  f_t\circ \gamma$ in $\calM([0,1]\times \sfL)$ \label{vf} 
        \item $\calG_t^n\circ\gamma^n  \rightharpoonup \calG_t\circ\gamma$ in $\calM([0,1]\times \Omega_\sfV)$ \label{vG}.
    \end{enumerate}
\end{proposition}

\begin{proof}
    By composition of the measures $\rho^n_t, \calV^n_t, f^n_t$ and $\calG^n_t$ with the change of variables maps $\gamma^n$ through the push-forward operation, the whole sequence of measures is defined on a common domain instead of varying graphs.  
    The remainder of the proof follows as in Theorem \ref{CompFixedG} applied to the resulting push-forward measures. The norm-convergece of the vertices implies $C^1([0,1])$-convergence of the maps $\gamma^n_e$ for all $e\in \sfE$, thus the change of variables can be reversed in the limit, concluding the proof.
\end{proof}

With compactness and lower semi-continuity at hand, we can prove well-posedness.

\begin{theorem}
    Suppose that \ref{VarFormVarG} is finite. Then, there exists an embedded star-shaped metric graph $\sfG=(\sfV,\sfE,l)$ and a family of measures $(\mu_t, J_t, \rho_t, V_t, f_t, G_t)\in \CE(\sfG)$ minimizing \ref{VarFormVarG}.
\end{theorem}

\begin{proof}
    Let $\sfG^n$ be a minimizing sequence of graphs with corresponding measures given by the sequence $(\mu^n, J^n, \rho^n, V^n, f^n, G^n)\in \CE(\sfG^n)$. Each graph is embedded with the embedding given by a family of vertices $\{x^n_{v_0},x^n_{v_1},\ldots,x^n_{v_{N_\sfV -1}}\}\subset\Omega$, defining the sequences $(x^n_{v_i})_{n\in\N}$ for $i\in\{0,\ldots, N_\sfV-1\}$. Compactness of $\Omega\subset\R^d$ allows to apply the Bolzano-Weierstrass theorem to each of these sequences. Therefore, there exists a family of limiting points $x_{v_i}\in\Omega$ for $i\in\{0,\ldots, N_\sfV-1\}$ such that, up to a subsequence, $x_{v_i}^n\to x_{v_i}$ for $i\in\{0,1,\ldots,N_\sfV-1\}$. Now Proposition \ref{CompVarG} and Proposition \ref{lscVarG} allow the application of the direct method, proving the theorem.
\end{proof}

\section{Acknowledgement}
 M.C.'s research was supported by the NWO-M1 grant Curve Ensemble Gradient Descents for Sparse Dynamic Problems (Grant Number OCENW.M.22.302).

\vspace{\baselineskip}

\bibliographystyle{alpha}
\bibliography{bibliography.bib}

\begin{thebibliography}{EFMM22}

\bibitem[AGS08]{ambrosio_gradient_2008}
Luigi Ambrosio, Nicola Gigli, and Guiseppe Savaré.
\newblock {\em Gradient {Flows}}.
\newblock Birkhäuser, Basel, 2008.

\bibitem[BHP23]{BurgerHumpertPietschmann2023}
Martin Burger, Ina Humpert, and Jan-Frederik Pietschmann.
\newblock Dynamic optimal transport on networks.
\newblock {\em {ESAIM}: Control, Optimisation and Calculus of Variations},
  29:54, 2023.

\bibitem[CGP05]{Coclite2005}
G.~M. Coclite, M.~Garavello, and B.~Piccoli.
\newblock Traffic flow on a road network.
\newblock {\em SIAM Journal on Mathematical Analysis}, 36(6):1862--1886, 2005.

\bibitem[CKP25]{carioni2025dynamicoptimaltransportoptimal}
Marcello Carioni, Juliane Krautz, and Jan-F. Pietschmann.
\newblock {Dynamic Optimal Transport with Optimal Preferential Paths}.
\newblock {\em arXiv preprint}, 2025.
\newblock arXiv:2504.03285.

\bibitem[DNS09]{dolbeault_new_2009}
Jean Dolbeault, Bruno Nazaret, and Giuseppe Savaré.
\newblock A new class of transport distances between measures.
\newblock {\em Calculus of Variations and Partial Differential Equations},
  34(2):193--231, February 2009.

\bibitem[EFMM22]{Erbar2022}
Matthias Erbar, Dominik Forkert, Jan Maas, and Delio Mugnolo.
\newblock Gradient flow formulation of diffusion equations in the wasserstein
  space over a metric graph.
\newblock {\em Networks and Heterogeneous Media}, 17(5):687--717, 2022.

\bibitem[FBP24]{Fazeny2024}
Ariane Fazeny, Martin Burger, and Jan-Frederik Pietschmann.
\newblock Optimal transport on gas networks.
\newblock {\em to appear in European Journal of Applied Mathematics}, 2024.

\bibitem[HPS24]{heinze_gradient_2024}
Georg Heinze, Jan-Frederik Pietschmann, and André Schlichting.
\newblock Gradient flows on metric graphs with reservoirs: {Microscopic}
  derivation and multiscale limits.
\newblock {\em arXiv preprint}, 2024.
\newblock arXiv:2412.16775.

\bibitem[LM10]{Lisini2010}
Stefano Lisini and Antonio Marigonda.
\newblock On a class of modified wasserstein distances induced by concave
  mobility functions defined on bounded intervals.
\newblock {\em manuscripta mathematica}, 133(1-2):197--224, June 2010.

\bibitem[Mon21]{monsaingeon_new_2021}
Léonard Monsaingeon.
\newblock A new transportation distance with bulk/interface interactions and
  flux penalization.
\newblock {\em Calculus of Variations and Partial Differential Equations},
  60(3), April 2021.

\end{thebibliography}

\end{document}